\documentclass[12pt, a4paper]{article}
\usepackage{geometry}
\usepackage{titlesec} 
\usepackage{amsfonts}
\usepackage{amsthm}
\usepackage{amssymb}
\usepackage{amsmath}
\usepackage{theoremref}
\usepackage{enumerate}
\usepackage{bbm}
\usepackage{bm}
\usepackage{authblk}
\usepackage{mathtools}
\usepackage{a4wide}
\usepackage{xcolor}

\DeclarePairedDelimiter\abs{\lvert}{\rvert}%
\DeclarePairedDelimiter\norm{\lVert}{\rVert}%

\makeatletter
\let\oldabs\abs
\def\abs{\@ifstar{\oldabs}{\oldabs*}}
\let\oldnorm\norm
\def\norm{\@ifstar{\oldnorm}{\oldnorm*}}
\makeatother

\makeatletter
\g@addto@macro\bfseries{\boldmath}
\makeatother

\newcommand{\M}{\mathcal{M}}

\newcommand{\T}{\mathbb{T}}

\newcommand{\conj}[1]{\overline{#1}}
\newcommand{\D}{\mathbb{D}}

\newcommand{\cD}{\conj{\mathbb{D}}}

\newcommand{\m}{\textit{m}}
 
\newcommand{\hd}{Hol(\D)}

\newcommand{\hb}{\mathcal{H}(b)}

\newtheorem{thm}{Theorem}[section]
\newtheorem{lemma}[thm]{Lemma}

\newtheorem{cor}[thm]{Corollary}
\newtheorem{prop}[thm]{Proposition}
\theoremstyle{definition}

\theoremstyle{definition}

\newcommand{\Addresses}{{
		\bigskip
		\footnotesize
		
		Bartosz Malman, \\ \textsc{KTH Royal Institute of Technology, \\
			Stockholm, Sweden}\\
			\texttt{malman@kth.se}		
	}}

\begin{document}
\title{\textbf{Cyclic inner functions in growth classes and applications to approximation problems}}
\author{Bartosz Malman}
\date{ }
\maketitle

\begin{abstract}
It is well-known that for any inner function $\theta$ defined in the unit disk $\D$ the following two conditons: $(i)$ there exists a sequence of polynomials $\{p_n\}_n$ such that $\lim_{n \to \infty} \theta(z) p_n(z) = 1$ for all $z \in \D$, and $(ii)$ $\sup_n \| \theta p_n \|_\infty < \infty$, are incompatible, i.e., cannot be satisfied simultaneously. In this note we discuss and apply a consequence of a result by Thomas Ransford from \cite{ransford2021decay}, which shows that if we relax the second condition to allow for arbitrarily slow growth of the sequence $\{ \theta(z) p_n(z)\}_n$ as $|z| \to 1$, then condition $(i)$ can be met. In other words, every growth class of analytic functions contains cyclic singular inner functions. We apply this observation to properties of decay of Taylor coefficients and moduli of continuity of functions in model spaces $K_\theta$. In particular, we establish a variant of a result of Khavinson and Dyakonov on non-existence of functions with certain smoothness properties in $K_\theta$, and we show that the classical Aleksandrov theorem on density of continuous functions in $K_\theta$, and its generalization to de Branges-Rovnyak spaces $\hb$, is essentially sharp. \end{abstract}

\section{Background and results}

Let $X$ be a topological space consisting of functions which are analytic in the unit disk $\D = \{ z \in \mathbb{C} : |z| < 1\}$ and which satisfies some customary desirable properties, such as that the evaluation $f \mapsto f(\lambda)$ is a continuous functional on $X$ for each $\lambda \in \D$, and that the function $z \mapsto zf(z)$ is contained in the space $X$ whenever $f \in X$. A function $g \in X$ is said to be \textit{cyclic} if there exists a sequence of analytic polynomials $\{p_n\}_n$ for which the polynomial multiples $\{gp_n\}_n$ converge to the constant function $1$ in the topology of the space. 

The well-known Hardy classes $H^p$ are among the very few examples of analytic function spaces in which the cyclicity phenomenon is completely understood. The cyclic functions $g$ are of the form \begin{equation} \label{outerdef}
    g(z) = \exp\Big( \int_\T \frac{\zeta + z}{\zeta - z} \log(|g(\zeta)|) \, d\m(\zeta)\Big), \quad z \in \D,
\end{equation} where $d\m$ is the (normalized) Lebesgue measure of the unit circle $\T = \{ z\in \mathbb{C} : |z| = 1\}$. Functions as in \eqref{outerdef} are called \textit{outer functions}. The \textit{inner functions} are of the form \begin{equation} \label{innerdef}
\theta(z) = \prod_{n} \frac{\conj{\alpha_n}}{|\alpha_n|}\frac{z - \alpha_n}{1-\conj{\alpha_n}z} \cdot \exp\Big( -\int_\T \frac{\zeta + z}{\zeta - z} d\nu(\zeta)\Big), \quad z \in \D,
\end{equation} where $\nu$ is a positive finite Borel measure on $\T$ and $\{\alpha_n\}_n$ is a Blaschke sequence. It is clear that if the Blaschke product on the left is non-trivial, then $\theta$ vanishes at points in $\D$ and therefore cannot be cyclic in any reasonable space of analytic functions $X$. The right factor is a \textit{singular inner function}, and it is well-known that if a function $g \in H^p$ has a singular inner function as a factor, then $g$ is not cyclic in $H^p$. As a consequence, if $\{p_n\}_n$ is a sequence of polynomials for which we have $$\lim_{n \to \infty} \theta(z)p_n(z) = 1, \quad z \in \D,$$ then necessarily the Hardy class norms of the sequence must explode: $$\lim_{n \to \infty} \, \|\theta p_n\|^p_{H^p} := \lim_{n \to \infty} \int_\T |\theta p_n|^p d\m = \infty$$ for finite $p \geq 1$, or in case $p = \infty$, $$ \lim_{n \to \infty} \|\theta p_n\|_\infty := \lim_{n \to \infty}\sup_{z \in \D} |\theta(z)p_n(z)| = \infty.$$
When other norms are considered, cyclic singular inner functions might exist, and here the Bergman spaces $L^p_a(\D)$ provide a famous set of examples. The Bergman norms are of the form $$\|g\|^p_{L^p(\D)} := \int_\D |g(z)|^p dA(z),$$ where $dA$ is the normalized area measure of $\D$. After a sequence of partial results by multiple authors, Boris Korenblum in \cite{korenblum1975extension} and James Roberts in \cite{roberts1985cyclic} independently characterized the cyclic singular inner functions in the Bergman spaces in terms of the vanishing on certain subsets of $\T$ of the corresponding singular measure $\nu$ appearing in \eqref{innerdef}. A construction of a singular inner function which is cyclic in the classical Bloch space appears in \cite{anderson1991inner}.

The purpose of this note is to discuss and apply a recent theorem of Thomas Ransford from \cite{ransford2021decay} which deals with singular inner functions which decay slowly near the boundary of the disk. Here is the statement of Ransford's theorem.

\begin{thm} \thlabel{pre-ransfordtheorem}
Let $w:[0,1) \to (0,1)$ be any function satisfying $\lim_{r \to 1^-} w(r) = 0$. Then there exists a singular inner function $\theta$ for which we have \begin{equation} \label{loweresttheta}
\min_{|z| < r} |\theta(z)| \geq w(r), \quad r \in (0,1).
\end{equation}
\end{thm}
We refer the reader to the article \cite{ransford2021decay} for the clever proof of the statement, which relies on an application of Frostman's classical construction of a measure with prescribed smoothness properties, and estimates of the Poisson kernel. We will apply Ransford's theorem in a slightly different form than it is stated in \cite{ransford2021decay}. The following consequence of his theorem is the pivotal point of our discussion. 

\begin{cor} \thlabel{ransfordtheorem} Let $w:[0,1] \to (0,1)$ be any decreasing function satisfying $\lim_{t \to 1} w(t) = 0$. There exists a singular inner function $\theta = S_\nu$ and a sequence of analytic polynomials $\{p_n\}_n$ such that 
\begin{enumerate}[(i)]
    \item $\lim_{n \to \infty} \theta(z) p_n(z) = 1, \quad z \in \D,$
    \item $\sup_{z \in \D} |\theta(z) p_n(z)| w(|z|) \leq 2$.
\end{enumerate}
\end{cor}

\begin{proof}
Apply Ransford's \thref{pre-ransfordtheorem} to the function $w$ to produce a singular inner function $\theta$ satisfying \eqref{loweresttheta}. For integers $n \geq 2$ we set $r_n := 1- 1/n$ and $Q_n(z) := 1/\theta(r_n z)$. Then $Q_n$ is holomorphic in a neighbourhood of the closed disk $\cD$, and because we are assuming that $w$ is decreasing, we have the estimate $$\sup_{z \in \D} |Q_n(z)|w(|z|) \leq \sup_{z \in \D} \frac{w(|z|)}{w(r_n|z|)} \leq 1.$$ We can approximate $Q_n$ by an analytic polynomial $p_n$ so that $$\sup_{z \in \D} |Q_n(z) - p_n(z)| \leq 1/n.$$
Then \begin{gather*}
    \sup_{z \in \D} |\theta(z) p_n(z)| w(|z|) \leq \sup_{z \in \D} \Big(|\theta(z) Q_n(z)| + 1/n \Big) w(|z|) \leq 2 
\end{gather*}
It is clear from the construction that $\theta(z)p_n(z) \to 1$ as $n \to \infty$, for any $z \in \D$.
\end{proof}
\thref{ransfordtheorem} says that there exists cyclic singular inner functions in essentially any space of analytic functions defined in terms of a growth condition, or in any space in which such a \textit{growth space} is continuously embedded. 

In the proofs of our main results, which will be stated shortly, we will concern ourselves with the following \textit{weak} type of cyclicity of singular inner functions. Let $Y$ be some linear space of analytic functions which is contained in $H^1$. We want to investigate if there exists a singular inner function $\theta$ and a sequence of polynomials $\{p_n\}_n$ such that \begin{equation}
    \label{weakcyclicity}
    f(0) = \int_\T f \, d\m = \lim_{n \to \infty} \int f \conj{\theta p_n} \,d\m
\end{equation} holds for all $f \in Y$. The above situation means that the sequence $\{\theta p_n\}_n$ converges to the constant $1$, weakly over the space $Y$. Now, clearly if $Y$ is too large of a space (say, $Y = H^2$), then \eqref{weakcyclicity} can never hold for all $f \in Y$. However, if $Y$ is sufficiently small, then the situation in \eqref{weakcyclicity} might occur. For instance, in the extreme case when $Y$ is a set of analytic polynomials, then any singular inner function $\theta$ and any sequence of polynomials $\{p_n\}_n$ which satisfies $\lim_{n \to \infty} p_n(z) = 1/\theta(z)$ for $z \in \D$, is sufficient to make \eqref{weakcyclicity} hold. Philosophically speaking, it is the uniform smoothness of the functions in the class $Y$ that allows the existence of singular inner functions $\theta$ for which the above situation occurs.

Such weak cyclicity of singular inner functions plays a role in the theory of approximations in model spaces $K_\theta$ and the broader class of de Branges-Rovnyak spaces $\hb$. Recall that the space $K_\theta$ is constructed from an inner function $\theta$ by taking the orthogonal complement of the subspace $$\theta H^2 := \{ \theta h : h \in H^2\}$$ in the Hardy space $H^2$: 
\begin{equation*}
    K_\theta = H^2 \ominus \theta H^2.
\end{equation*}
For background on the spaces $K_\theta$ one can consult the books \cite{cauchytransform} and \cite{garcia2016introduction}. 

Note that the integrals on the right-hand side in \eqref{weakcyclicity} represents the inner product in $H^2$ between a function $f$ and a polynomial multiple of $\theta$. Under insignificat assumptions on $Y$, a straightforward argument will show that if $\eqref{weakcyclicity}$ occurs, then the intersection between $Y$ and $K_\theta$ is trivial: $Y \cap K_\theta = \{0\}$. As a consequence of these observations, Ransford's theorem, and some further analysis, we will be able to construct model spaces $K_\theta$ which have trivial intersections with many examples of spaces $Y$, even ones defined by very mild smoothness or continuity conditions. Consequently, we will be able to construct spaces $K_\theta$ in which every non-zero function behaves \textit{badly} in some precise way.

In one of our main results, we will show that the famous approximation theorem of Aleksandrov from \cite{aleksandrovinv} on density in $K_\theta$ of functions which extend continuously to the boundary, is in fact essentially sharp, as it cannot be extended to any class of functions satisfying an estimate on their modulus of continuity. By a \textit{modulus of continuity} $\omega$ we mean a function $\omega: [0,\infty) \to [0, \infty)$ which is continuous, increasing, satisfies $\omega(0) = 0$, and for which $\omega(t)/t$ is a decreasing function with $$\lim_{t \to 0^+} \omega(t)/t = \infty.$$ For such a function $\omega$ we define $\Lambda^\omega_a$ to be the space of functions $f$ which are analytic in $\D$, extend continuously to $\cD$, and satisfy \begin{equation}\label{mocspacedef}
\sup_{z,w \in \cD, z \neq w} \frac{|f(z)-f(w)|}{\omega(|z-w|)} < \infty.
\end{equation} Then $\Lambda^\omega_a$ is the space of analytic functions on $\D$ which have a modulus of continuity dominated by $\omega$. We make $\Lambda^\omega_a$ into a normed space by introducing the quantity $$\|f\|_\omega := \|f\|_\infty + \sup_{z,w \in \cD, z \neq w} \frac{|f(z)-f(w)|}{\omega(|z-w|)}.$$ 
By a theorem of Tamrazov from \cite{tamrazov1973contour}, we could have replaced the supremum over $\cD$ by a supremum over $\T$, and obtain the same space of functions (we remark that a nice proof of this result is contained in \cite[Appendix A]{bouya2008closed}). One of our two main results is the following optimality statement regarding Aleksandrov's density theorem.

\begin{thm} \thlabel{mocnondensitymodelspaces}
Let $\omega$ be a modulus of continuity. There exists a singular inner function $\theta$ such that $$\Lambda^\omega_a \cap K_\theta = \{0\}.$$ 
\end{thm}
This statement will be proved in Section \ref{aleksandrovoptimality}. In fact, we will see that \thref{mocnondensitymodelspaces} is a consequence of a variant, and in some directions a strengthening, of a theorem of Dyakonov and Khavinson from \cite{starinvsmooth}. For a sequence of positive numbers $\bm{\lambda} = \{\lambda_n\}_{n=0}^\infty $ we define the class \begin{equation}
    \label{H2wdef}
    H^2_{\bm{\lambda}} = \Big\{ f = \sum_{n=0}^\infty f_n z^n \in \hd : \sum_{n=0}^\infty \lambda_n|f_n|^2 < \infty \Big\}.
\end{equation}

The other of our main results, proved in Section \ref{badmodelspacesection}, reads as follows.

\begin{thm} \thlabel{badmodelspace}
Let $\bm{\lambda} = \{\lambda_n\}_{n=0}^\infty$ be any increasing sequence of positive numbers with $\lim_{n \to \infty} \lambda_n = \infty$. Then there exists a singular inner function $\theta$ such that $$K_\theta \cap H^2_{\bm{\lambda}} = \{0\}.$$
\end{thm}

The result can be compared to the mentioned result of Dyakonov and Khavinson in \cite{starinvsmooth}, from which the above result can be deduced in the special case $\bm{\lambda} = \{ (k+1)^{\alpha}\}_{k=0}^\infty$ with any $\alpha > 0$. 

The theory of de Branges-Rovnyak spaces $\hb$ is a well-known generalization of the theory of model spaces $K_\theta$. The symbol of the space $b$ is now any analytic self-map of the unit disk, and we have $\hb = K_b$ whenever $b$ is inner. For background on $\hb$ spaces, see \cite{sarasonbook}, or \cite{hbspaces1fricainmashreghi} and \cite{hbspaces2fricainmashreghi}. A consequence of the authors work in collaboration with Alexandru Aleman in \cite{dbrcont} is that the above mentioned density theorem of Aleksandrov generalizes to the broader class of $\hb$ spaces: any such space admits a dense subset of functions which extend continuously to the boundary. Since \thref{mocnondensitymodelspaces} proves optimality of Aleksandrov's theorem for inner functions $\theta$, one could ask if at least for outer symbols $b$ any improvement of the density result in $\hb$ from \cite{dbrcont} can be obtained. In Section \ref{hboptimalitysec} we remark that this is not the case, and the result in \cite{dbrcont} is also essentially optimal, even for outer symbols $b$. This is shown to be equivalent to a theorem of Khrushchev from \cite{khrushchev1978problem}. 

In the last Section \ref{questionsec} we list a few questions we have not found an answer for, and some ideas for further research. 

\section{Proof of \thref{badmodelspace}}

\label{badmodelspacesection}

In the proof of the theorem we will need to use the following crude construction of an integrable weight with large moments.

\begin{lemma} \thlabel{momentlemma}
Let $\{\lambda_n\}_{n=0}^\infty$ be a decreasing sequence of positive numbers with $\lim_{n \to \infty} \lambda_n = 0$. There exists a non-negative function $\Lambda \in L^1([0,1])$ which satisfies $$\lambda_n \leq \int_{0}^1 x^{2n+1} \Lambda(x) \,dx, \quad n \geq 0.$$
\end{lemma}

\begin{proof}
Recall that the sequence $(1-1/n)^n = \exp(n\log(1-1/n))$ is decreasing and satisfies $$\lim_{n \to \infty} (1-1/n)^n = e^{-1}.$$ It follows that $$\inf_{x \in (1-1/n, 1)} x^{2n+1} \geq \alpha$$ for some constant $\alpha > 0$ which is independent of $n$. For $n \geq 1$, we define the intervals $I_n = \big( 1 - 1/n, 1 - 1/(n+1) \big)$. Our function $\Lambda$ will be chosen to be of the form $$\Lambda(x) = \sum_{n=0}^\infty 1_{I_n} c_n,$$ where $1_{I_n}$ is the indicator function of the interval $I_n$ and the $c_n$ are positive constants to be chosen shortly. Note that \begin{gather} \label{firstest} \int_0^1 x^{2N+1} \Lambda(x) \, dx \geq \int_{1-1/N}^1 x^{2N+1} \Lambda(x) \, dx \geq \alpha \sum_{n = N}^\infty |I_n| c_n. \end{gather} We choose $$c_n = \alpha^{-1}|I_n|^{-1}(\lambda_n - \lambda_{n+1}).$$ This choice of coefficients $c_n$ makes $\Lambda$ integrable over $[0,1]:$ \begin{gather*} 
\int_{0}^1 \Lambda(x) dx = \sum_{n=1}^\infty |I_n|c_n = \alpha^{-1} \sum_{n=1}^\infty \lambda_n - \lambda_{n+1} \\ =
\lim_{M \to \infty} \alpha^{-1} \sum_{n=1}^M \lambda_n - \lambda_{n+1} = \lim_{M \to \infty} \alpha^{-1} (\lambda_1 - \lambda_{M+1}) \\ = \alpha^{-1} \lambda_1 \end{gather*} In the last step we used the assumption that the sequence $\{\lambda_n\}_n$ converges to zero. Moreover, by \eqref{firstest} and the choice of $c_n$ we can estimate 
\begin{gather*}\int_0^1 x^{2N+1} \Lambda(x) \, dx \geq \alpha \sum_{n=N}^\infty |I_n|c_n \\ = \lim_{M \to \infty} \alpha \sum_{n=N}^M |I_n| c_n = \lim_{M \to \infty} \sum_{n=N}^\M \lambda_n -\lambda_{n+1} \\ = \lim_{M \to \infty} \lambda_N - \lambda_{M+1}  = \lambda_N.\end{gather*}  The proof is complete. 
\end{proof}

The significance of the above lemma is the estimate \begin{equation} \label{Lambdanorm}
    \sum_{k=0}^\infty \lambda_k|f_k|^2 \leq c \int_\D |f(z)|^2 \Lambda(|z|) dA(z)
\end{equation} for some numerical constant $c > 0$ and any function $f$ which is holomorphic in a neighbourhood of the closed disk $\cD$. The estimate can be vereified by direct computation of the integral on the right-hand side, using polar coordinates.

We will also use the following well-known construction.

\begin{lemma} \thlabel{multiplier}
For any function $g \in L^1([0,1])$ there exists a positive and increasing function $w: [0,1) \to \mathbb{R}$ which satisfies $$\lim_{t \to 1^-} w(t) = \infty$$ and $$wg \in L^1([0,1]).$$
\end{lemma}

\begin{proof} The integrability condition on $g$ implies that $$\lim_{t \to 1^-} \int_{t}^1 |g(x)| dx = 0.$$
Thus there exists a sequence of intervals $\{I_n\}_{n=1}^\infty$ which have $1$ as the right end-point and which satisfy $I_{n+1} \subset I_n$ for all $n \geq 1$, and $$\int_{I_n} |g(x)| dx \leq 4^{-n}.$$ If we set $$w(t) = 1_{[0,1) \setminus I_1} +  \sum_{n=1}^\infty 2^n 1_{I_n \setminus I_{n+1}},$$ where $1_{I_n \setminus I_{n+1}}$ is the indicator function of the set difference $I_n \setminus I_{n+1}$, then $w$ is increasing, satisfies $\lim_{t \to 1^-} w(t) = \infty$, and $$\int_{I_n \setminus I_{n+1}} w(x)|g(x)| dx \leq 2^{-n}$$ for all $n \geq 1$. Consequently $$\int_{0}^1 w(x)|g(x)|dx \leq \int_{0}^1 |g(x)|dx + \sum_{n} \int_{I_n \setminus I_{n+1}} w(x)|g(x)| dx < \infty.$$
\end{proof}

\begin{proof}[Proof of \thref{badmodelspace}]
Let $\Lambda$ be the function in \thref{momentlemma} which corresponds to the sequence $\{1/\lambda_n\}_{n=0}^\infty$. That is, $\Lambda$ satisfies $$\frac{1}{\lambda_n} \leq \int_0^1 x^{2n+1} \Lambda(x) \, dx, \quad n \geq 0,$$ and $\Lambda \in L^1[0,1]$. Now, let $w$ be a positive decreasing function which satisfies $\lim_{x \to 1^-} w(x) = 0$ and $$\int_0^1 \frac{\Lambda(x)}{w^2(x)} \,dx < \infty.$$ Existence of such a function follows readily from \thref{multiplier}. Apply \thref{ransfordtheorem} to $w$ and obtain a corresponding inner function $\theta$ and a sequence of polynomials $\{p_n\}_n$ for which the conclusions $(i)$ and $(ii)$ of \thref{ransfordtheorem} hold. We will show that for this $\theta$ we have $K_\theta \cap H^2_{\bm{\lambda}} = \{0\}.$

Indeed, assume $f \in K_\theta \cap H^2_{\bm{\lambda}} = \{0\},$ but that in fact $f$ is non-zero. Since both $K_\theta$ and $H^2_{\bm{\lambda}}$ are invariant for the backward shift operator, we may without loss of generality assume that $f(0) \neq 0$. Fix an integer $n$ and let \begin{equation}
    \label{gdef} g(z) = \theta(z)p_n(z) - 1, z \in \D.
\end{equation} Let $\{f_k\}_k, \{g_k\}_k$ be the sequences of Taylor coefficients of $f$ and $g$, respectively. Since $f \in K_\theta$, we have
\begin{gather*}
    |f(0)| = \Big\vert \int_\T f \, d\m \Big\vert =  \Big\vert \int_\T f\conj{\theta p_n - 1} \, d\m\Big\vert = \lim_{r \to 1^-} \Big\vert \sum_{k=0}^\infty r^{2k} f_k \conj{g_k} \Big\vert \\ \leq \limsup_{r \to 1^-}\Big( \sum_{k=0}^\infty \lambda_k r^{2k} |f_k|^2 \Big)^{1/2} \Big( \sum_{k=0}^\infty \frac{1}{\lambda_n} |r^k g_k|^2 \Big)^{1/2} \\ 
\end{gather*}
Using inequality \eqref{Lambdanorm} on the term on the right-hand side in the last expression (with $\lambda_n$ replaced by $1/\lambda_n$), we obtain

\begin{gather*}
    |f(0)| \leq  C \limsup_{r \to 1^-} \Big( \sum_{k=0}^\infty \lambda_k  |f_k|^2 \Big)^{1/2} \Big(\int_\D |g(rz)|^2 \Lambda(|z|)d A(z) \Big)^{1/2} \\ = C \Big( \sum_{k=0}^\infty \lambda_k  |f_k|^2 \Big)^{1/2} \Big(\int_\D |g(z)|^2 \Lambda(|z|)d A(z) \Big)^{1/2}.
\end{gather*}
By assertion in part $(ii)$ of \thref{ransfordtheorem}, the function $|g(z)|^2\Lambda(|z|)$ is dominated pointwise in $\D$ by the integrable function $$\frac{4\Lambda(|z|)}{w^2(|z|)}, \quad z \in \D$$ independently of which polynomial $p_n$ is used to defined $g$ in \eqref{gdef}. But if we let $n \to \infty$ in \eqref{gdef}, then $|g(z)|^2\Lambda(|z|) \to 0$, and so we infer from the computation above and the dominated convergence theorem that $f(0) = 0$, which is a contradiction. The conclusion is that $K_\theta \cap H^2_{\bm{\lambda}} = \{ 0\}$, and the proof of the theorem is complete.
\end{proof}

\section{Proof of \thref{mocnondensitymodelspaces}}

\label{aleksandrovoptimality}

\thref{mocnondensitymodelspaces} will follow immediately from \thref{badmodelspace} together with the following embedding result for the spaces $\Lambda^\omega_a$.

\begin{lemma} \thlabel{mocspecestimatelemma}
Let $\omega$ be a modulus of continuity. There exists an increasing sequence of positive numbers $\bm{\alpha} = \{ \alpha_n\}_{n=0}^\infty$ satisfying $\lim_{n \to \infty} \alpha_n = \infty$ and such that for any $f \in \Lambda^\omega_a$ we have the estimate \begin{equation} \label{specembedding}
\sum_{n=0} \alpha_n|f_n|^2 \leq C \|f\|^2_\omega
\end{equation} where $C > 0$ is a numerical constant and $\{f_n\}_n$ is the sequence of Taylor coefficients of $f$.
\end{lemma}

\begin{proof}
For each $r \in (0,1)$ we have the estimate \begin{equation} \label{valuestospectrum}
\sum_{n=0} (1-r^{2n})|f_n|^2 = \int_{\T} |f(\zeta) - f(r\zeta)|^2 d\m(\zeta) \leq \omega(1-r)^2 \|f\|^2_\omega.
\end{equation} Since $\lim_{t \to 0} \omega(t) = 0$, for each positive integer $N$ there exists a number $r_N \in (0,1)$ such that $\omega(1-r_N) \leq \frac{1}{2^N}$. Since $\lim_{n \to \infty} r_N^{2n} = 0$, there exists an integer $K(N)$  such that $r_N^{2n} < 1/2$ for $n \geq K(N)$. Then \begin{equation*}
    \sum_{n = K(N)}^\infty \frac{|f_n|^2}{2} \leq \sum_{n=K(N)}^\infty (1-r_N^{2n})|f_n|^2 \leq \frac{1}{4^N}\|f\|^2_\omega. 
\end{equation*} Consequently \begin{equation} \label{specpartialestimate}
    \sum_{n=K(N)} 2^N |f_n|^2 \leq \frac{1}{2^{N-1}} \|f\|^2_\omega.
\end{equation} We can clearly choose the sequence of integers $K(N)$ to be increasing with $N$. If we define the sequence $\bm{\alpha}$ by the equation $\alpha_n = 1$ for $n \leq K(1)$, and $\alpha_n = 2^n$ for $K(N) \leq n < K(N+1)$, then \eqref{specembedding} follows readily from \eqref{specpartialestimate} by summing over all $N \geq 1$.
\end{proof}

\begin{proof}[Proof of \thref{mocnondensitymodelspaces}]
\thref{mocspecestimatelemma} implies that $\Lambda^\omega_a$ is contained in some space of the form $H^2_{\bm{\alpha}}$ as defined in \eqref{H2wdef}. If $\theta$ is a singular inner function given by \thref{badmodelspace} such that $H^2_{\bm{\alpha}} \cap K_\theta = \{0\}$, then obviously we also have that $\Lambda^\omega_a \cap K_\theta = \{0\}$, and so the claim follows.
\end{proof}

\section{The case of $\hb$ spaces}

\label{hboptimalitysec}

Here we prove the optimality of the continuous approximation theorem for the larger class of $\hb$ spaces.

\begin{prop}
Let $\omega$ be a modulus of continuity. There exists an outer function $b:\D \to \D$ such that $$\Lambda^\omega_a \cap \hb = \{ 0 \}.$$
\end{prop}

\begin{proof}
By a result of Khrushchev noted in \cite[Theorem 2.4]{khrushchev1978problem}, there exists a closed subset $E$ of the circle $\T$ with the property that for no non-zero integrable function $h$ supported on $E$ is the Cauchy integral $$C_h(z) = \int_{\T} \frac{h(\zeta)}{1-z\conj{\zeta}} d\m(\zeta)$$ a member of the space $\Lambda^\omega_a$. It suffices thus to construct an $\hb$ space for which every function can be expressed as such a Cauchy integral. The simplest choice for the space symbol $b$ is the outer function with modulus $1$ on $\T \setminus E$ and $1/2$ on $E$. Then $b$ is invertible in the algebra $H^\infty$, and a consequence of the general theory (see \cite[Theorem 20.1 and Theorem 28.1] {hbspaces2fricainmashreghi}) is that every function in the space $\hb$ is a Cauchy integral of a function $h$ which is square-integrable on $\T$ and supported only on $E$. Thus $\Lambda^\omega_a \cap \hb = \{ 0 \}$. 
\end{proof}

\section{Some ending questions and remarks}

\label{questionsec}

Since Ransford's theorem seems to be such a powerful tool in establishing results of the kind mentioned here, we are wondering whether it can be further applied. In particular, the following questions come to mind.

\begin{enumerate}
    \item Are our methods strong enough to prove that there exists model spaces $K_\theta$ which admit no non-zero functions in the Wiener algebra of absolutely convergent Fourier series? The result is known, and has been noted in \cite{limani2021abstract}. However, it was proved as a consequence of a complicated construction of a cyclic singular inner function in the Bloch space. Is it so that Ransford's construction is sufficient to prove the non-density result for the Wiener algebra in the fashion presented here?
    
    \item For $p > 2$, the Banach spaces $\ell^p_a$ consisting of functions $f\in \hd$ with Taylor series $\{f_n\}_{n=0}^\infty$ satisfying $$\|f\|^p_{\ell^p_a} := \sum_{n=0}^\infty |f_n|^p < \infty$$ are of course larger than the space $H^2 = \ell^2_a$. Do there exist cyclic singular inner functions in these spaces?
\end{enumerate}

\subsubsection*{Acknowledgements} The author would like to thank Christopher Felder for his reading of the manuscript and for his helpful suggestions for improvements. He would also like to thank Adem Limani for very useful discussions and important comments.

\bibliographystyle{siam}
\bibliography{mybib}

\Addresses

\end{document}